\documentclass{proc-l-hijacked}
\usepackage{amssymb, amsmath, amsthm, hyperref,}

\newtheorem{theorem}{Theorem}[section]
\newtheorem{lemma}[theorem]{Lemma}

\theoremstyle{definition}

\newtheorem{example}[theorem]{Example}

\numberwithin{equation}{section}

\DeclareMathOperator{\dist}{dist}

\begin{document}

	\title[]{Geometry of the Spectral Semidistance in Banach algebras}
	\author{G. Braatvedt and R. Brits}
	\address{Department of Mathematics, University of Johannesburg, South Africa}
	\email{gabraatvedt@uj.ac.za, rbrits@uj.ac.za}

\begin{abstract}
Let $A$ be a unital Banach algebra over $\mathbb C$, and suppose that the nonzero spectral values of, respectively, $a\mbox{ and }b\in A$ are discrete sets which cluster at $0\in\mathbb C$, if anywhere. We develop a plane geometric formula for the spectral semidistance of $a$ and $b$ which depends on the two spectra, and the orthogonality relationships between the corresponding sets of Riesz projections associated with the nonzero spectral values. Extending a result of Brits and Raubenheimer, it is further shown that $a$ and $b$ are quasinilpotent equivalent if and only if all the Riesz projections, $p(\alpha,a)$ and $p(\alpha,b)$,  correspond. For certain important classes of decomposable operators (compact, Riesz, etc.) the proposed formula reduces the involvement of the underlying Banach space $X$ in the computation of the spectral semidistance, and appears to be a useful alternative to Vasilescu's geometric formula (which requires knowledge of the local spectra of the operators at each $0\not=x\in X$). The apparent advantage gained through the use of a global spectral parameter in the formula aside, the various methods of complex analysis could then be employed to deal with the spectral projections; we give examples illustrating the utility of the main results.
\end{abstract}
	\parindent 0mm
	
	\maketitle

\section{Introduction}
 Let $A$ denote a complex Banach algebra with identity $\mathbf 1$. For $a,b\in A$ associate operators $L_a$, $R_b$, and $C_{a,b}$, acting on $A$, by
the relations
\[L_ax=ax, \quad R_bx=xb,\quad \hbox{and}\quad C_{a,b}x=(L_a-R_b)x\ \mbox{ for each } x\in A.\]
Since $L_a$ and $R_b$ commute it is easy that
\[C_{a,b}^nx=\sum_{k=0}^n(-1)^k{n \choose k}a^{n-k}xb^k\ \mbox{ for each } x\in A,\] with the convention that, if $0\not=a\in A$,
then $a^0=\mathbf 1$.
Using the particular value $x=\mathbf 1$, define $\varrho:A\times A\rightarrow \mathbb R$ by
\begin{equation}\label{varrho}
\varrho (a,b)=\limsup_n\left\| C_{a,b}^n\mathbf 1\right\| ^{1/n},
\end{equation}
 and then define
 \begin{equation}\label{rho}
\rho(a,b)=\sup \{\varrho (a,b),\varrho (b,a)\}.
\end{equation}

  If $X$ is a Banach space, and $A=\mathcal L(X)$, the Banach algebra of bounded linear operators from $X$ into $X$, then the number $\varrho(S,T)$ is a well-established quantity called the \emph{local spectral radius} \cite[p.235]{l+n} of the commutator $C_{S,T}\in\mathcal L(A)$ at $I$. The number $\rho(S,T)$ is called the \emph{spectral distance} \cite[p.251]{l+n} of the operators $S$ and $T$. Furthermore, the pair $(S,T)$ is said to be \emph{asymptotically intertwined} \cite[p.248]{l+n} by the identity, $I$, if $\varrho(S,T)=0$. If each of the pairs $(S,T)$ and $(T,S)$ is asymptotically intertwined by the identity operator (i.e. $\rho(S,T)=0$), then $S$ and $T$ are called \emph{quasinilpotent equivalent} \cite[p.253]{l+n}. A first generalization in the framework
of Banach algebras on topics related to the commutator appeared in Section III.4 of the monograph \cite{vas3}.
In the paper \cite{razpet} $\rho$ is called the \emph{spectral semidistance} which is perhaps a little more appropriate in view of the fact that $\rho$ is only a semimetric \cite[Proposition 3.4.9]{l+n}. One may think of the spectral semidistance as a noncommutative generalization of the distance induced by the spectral radius when $a$ and $b$ do commute.  Again, if $\rho(a,b)=0$, then $a$ and $b$ are said to be \emph{quasinilpotent equivalent.}
 A good source of results on the topic of the spectral (semi)distance is Laursen and Neumann's recent monograph \cite{l+n}; the Reader may also want to look at \cite{b+r, c+f, f+v, razpet, vas1, vas2}. We should mention the following simple but useful property of $\varrho$ and $\rho$ which appears explicitly in \cite[Lemma 2.2]{b+r}: If $q_a$ and $q_b$ are quasinilpotent elements of $A$ commuting with, respectively, $a$ and $b$, then $\varrho(a,b)=\varrho(a+q_a,b+q_b)$.

 The results in the present paper are related to Vasilescu's geometric formula \cite{vas2} for the spectral semidistance of decomposable operators $S,T\in\mathcal L(X)$:
  \begin{equation*}
  \rho(S,T)=\sup\{\max\{\dist(\lambda,\sigma_T(x)),\dist(\mu,\sigma_S(x))\}:x\not=0,\lambda\in\sigma_S(x),\mu\in\sigma_T(x)\}
  \end{equation*}
where $\sigma_S(x)$ and $\sigma_T(x)$ are, respectively, the local spectra of $S$ and $T$ at $x\in X$.

 The usual spectrum of $a\in A$ will be denoted by $\sigma (a,A)$, the ``nonzero" spectrum, $\sigma (a,A)\backslash\{0\}$, by $\sigma^\prime(a,A)$, and the spectral radius of  $a\in A$ by $r_\sigma(a,A)$. Whenever there is no ambiguity we shall omit the $A$ in $\sigma$ and $r_\sigma$.

If $a\in A$ and $\alpha\in\mathbb C$ is not an accumulation point of $\sigma(a)$, then let $\Gamma_{\alpha}$ be a small circle, disjoint from $\sigma(a)$, and isolating $\alpha$ from the remaining spectrum of $a$. We denote by
 \[p(\alpha,a)=\frac{1}{2{\pi}i}\int_{\Gamma_{\alpha}}(\lambda\mathbf 1 -a)^{-1}\,d\lambda\] the \emph{Riesz projection associated with $a$ and $\alpha$}. If $\alpha\notin\sigma(a)$, then, by Cauchy's Theorem, $p(\alpha,a)=0$. For Riesz projections $p(\alpha_1,a)\mbox{ and }p(\alpha_2,a)$, with $\alpha_1\not=\alpha_2$, the Functional Calculus implies that $p(\alpha_1,a)p(\alpha_2,a)=p(\alpha_2,a)p(\alpha_1,a)=0$.

We recall the following well-known ``spectral decomposition" (see \cite[p.21]{alexander}) from the Theory of Banach algebras:

\begin{lemma}\label{rep} Suppose $a\in A$ has $\sigma(a)=\{\lambda_1,\dots,\lambda_n\}$. Then $a$ has the representation
\[a=\lambda_1p_1+\cdots+\lambda_np_n+r_a\] where $p_i=p(\lambda_i,a)$, $\sum p_i=\mathbf 1$, and $r_a$ is a quasinilpotent element belonging to the
bicommutant of $a$.
\end{lemma}

 It is worthwhile to mention here a curious connection which relates $\varrho$ to the growth characteristics of a certain entire map from $\mathbb C$ into $A$: Let $f$ be an entire $A$-valued function. Then $f$ has an everywhere convergent power series expansion
\[f(\lambda)=\sum_{n=0}^\infty a_n\lambda^n,\] with coefficients $a_n$ belonging to $A$. Define a function $M_f(r)=\sup\limits_{|\lambda|\leq r}\|f(\lambda)\|$, $r>0$. The function $f$ is said to be of \emph{finite order} if there exists $K>0$ and $R>0$ such that $M_f(r)<e^{r^K}$ holds for all $r>R$. The infimum of the set of positive real numbers, $K$, such that the preceding inequality holds is called the \emph{order} of $f$, denoted by $\omega_f$. If $\omega_f=1$ then $f$ is said to be of \emph{exponential order}. Suppose $f$ is entire, and of finite order $\omega:=\omega_f$. Then $f$ said to be of \emph{finite type} if there exists $L>0$ and $R>0$ such that $M_f(r)<e^{Lr^{\omega}}$ holds for all $r>R$. The infimum of the set of positive real numbers, $L$, such that the preceding inequality holds is called the \emph{type} of $f$, denoted by $\tau_f$. It it known (see the monograph \cite[p.41]{levin}) that the order and type are given by the formulas
\[\omega_f=\limsup_n\left(\frac{n\log n}{\log\|a_n\|^{-1}}\right)\mbox{ and }\tau_f=\frac{1}{e\omega_f}\limsup_n\left(n\sqrt[n]{{\|a_n\|^{\omega_f}}}\right).\]
Concerning the formula for $\tau_f$, we remark that if $f$ is of order $0$ and finite type, then it follows directly from the definition, together with Liouville's Theorem, that $f$ must be constant.

 Let $a,b\in A$, and define
\[f:\lambda\mapsto e^{\lambda a}e^{-\lambda b},\hskip 0,5 cm \lambda\in\mathbb C.\] The corresponding series expansion, valid for all $\lambda\in\mathbb C$, is given by
\[f(\lambda)=e^{{\lambda}a}e^{-{\lambda}b}=\sum_{n=0}^\infty\frac{{\lambda}^nC_{a,b}^n\mathbf 1}{n!}.\]
Since $\|f(\lambda)\|\leq e^{\left(\|a\|+\|b\|\right)|\lambda|}$, for all $\lambda\in\mathbb C$, it is immediate, from the definition, that $f$ is of order at most one. Suppose we know that $f$ is of exponential order (i.e. $\omega_f=1$). Recall now, using Stirling's formula, that $\lim_nn(1/n!)^{1/n}=e$, from which we subsequently obtain
 \[\tau_f=\frac{1}{e}\limsup_n\left(n(1/n!)^{1/n}\left\| C_{a,b}^n\mathbf 1\right\| ^{1/n}\right)=\varrho(a,b).\] To start with, we give a really brief argument, using these ideas, which quickly leads to (an improvement of) the main result in Section $4$ of \cite{b+r}.

\begin{theorem}\label{growth} If $\sigma(a)$ and $\sigma(b)$ are finite, then $\varrho(a,b)=0$ if and only if $a-r_a=b-r_b$ where $r_a$ and $r_b$ are quasinilpotent elements commuting with $a$ and $b$ respectively.
\end{theorem}
\begin{proof}
The reverse implication is trivial as in \cite{b+r}. With Lemma~\ref{rep} we can write  $a-r_a=\sum_{j=1}^n\lambda_jp_j$ and $b-r_b=\sum_{j=1}^k\beta_jq_j$. Denote $\bar a=a-r_a$, $\bar b=b-r_b$, and define $f(\lambda)=e^{\lambda \bar a}e^{-\lambda \bar b}$.  Notice, since $\sum_{j=1}^np_j=\mathbf 1$ and $\sum_{j=1}^kq_j=\mathbf 1$, and using the orthogonality, we have

\begin{equation}\label{finitespec}
f(\lambda)=\left[\mathbf 1+\sum_{j=1}^n(e^{\lambda_j\lambda}-1)p_j\right]\left[\mathbf 1+\sum_{j=1}^k(e^{-\beta_j\lambda}-1)q_j\right]
=\sum\limits_{i,j}e^{(\lambda_i-\beta_j)\lambda}p_iq_j.
\end{equation}

 Fix any $i\in\{1,\dots,n\},\,j\in\{1,\dots,k\}$ such that $p_iq_j\not=0$, and define \[g_{i,j}(\lambda)=p_if(\lambda)q_j=e^{(\lambda_i-\beta_j)\lambda}p_iq_j.\] Let us assume $\lambda_i\not=\beta_j$. If we notice, using Stirling's formula, that $\lim_n\frac{n\log n}{\log n!}=1$, then the coefficient formula for the order applied to
   the representation $g_{i,j}(\lambda)=e^{(\lambda_i-\beta_j)\lambda}p_iq_j$ shows that $g_{i,j}$ is of exponential order.
  But now, on the one hand, using the submultiplicative norm inequality, the representation \[g_{i,j}(\lambda)=\sum_{n=0}^\infty\frac{{\lambda}^np_i\left(C_{\bar a,\bar b}^n\mathbf 1\right)q_j}{n!}\] gives the type of $g_{i,j}$ as $\varrho(\bar a,\bar b)=\varrho(a,b)=0$, and on the other, the representation $g_{i,j}(\lambda)=e^{(\lambda_i-\beta_j)\lambda}p_iq_j$ says the type is equal to
  $|\lambda_i-\beta_j|\not=0$. From this contradiction we may conclude that, for each pair $i,j$, either $p_iq_j=0$ or $\lambda_i=\beta_j$.
  It then follows from \eqref{finitespec} that $f$ is constant, so that $f(\lambda)=e^{\lambda \bar a}e^{-\lambda \bar b}=\mathbf 1$ for all $\lambda\in\mathbb C$. Differentiation finally gives $\bar a=\bar b$.
\end{proof}

\section{Geometry of $\varrho$}

To obtain the main result, Theorem~\ref{charf}, we first need to establish the formula in the case where $\sigma(a)$ and $\sigma(b)$ are finite sets.
As in the proof of Theorem~\ref{growth}, using Lemma~\ref{rep}, we can then write
$a=\sum_{i=1}^n\lambda_ip_i+r_a$ and $b=\sum_{j=1}^k\beta_jq_j+r_b$. Setting $\bar a=a-r_a$ and $\bar b=b-r_b$, we obtain the following:

\begin{lemma}\label{spec3} Suppose  $\sigma(a)$ and $\sigma(b)$ are finite. Then there exists a finite dimensional Banach space $X\subseteq A$ such that $\varrho(a,b)=r_\sigma(L_{\bar a}-R_{\bar b},\mathcal L(X))$.
\end{lemma}
\begin{proof} Let $X$ denote the normed space spanned by
the set \[Y=\{p^r_iq^t_j:i\in\{1,\dots,n\},j\in\{1,\dots,k\},r\in\{0,1\},t\in\{0,1\}\}.\] It is elementary that $L_{\bar a}$ and
 $R_{\bar b}$ belong to $\mathcal L(X)$.
Without loss of generality we may assume $Y$ constitutes a linearly independent set of vectors. Since $X$ has finite dimension there exist
$K_1,K_2>0$ such that if $x$ is a linear combination of elements in $Y$ with coefficients $\gamma_0,\dots,\gamma_s$ then
\[K_1(|\gamma_0|+\cdots+|\gamma_s|)\leq\|x\|\leq K_2(|\gamma_0|+\cdots+|\gamma_s|).\] Obviously we may take $K_2$ as
\[K_2=\sup\{\|p_i\|\,\|q_j\|+1:i\in\{1,\dots,n\},j\in\{1,\dots,k\}\}.\] So for $x\in X$ given by, say
\[x=\gamma_0\mathbf 1+\gamma_1p_1+\gamma_2q_1+\gamma_3p_1q_1+\cdots+\gamma_sp_nq_k\]
it follows that
\[C^m_{\bar a,\bar b}x=\gamma_0[C^m_{\bar a,\bar b}\mathbf 1]+\gamma_1p_1[C^m_{\bar a,\bar b}\mathbf 1]+\gamma_2[C^m_{\bar a,\bar b}\mathbf 1]q_1+\cdots+\gamma_sp_n[C^m_{\bar a,\bar b}\mathbf 1]q_k,\] and thus that
\begin{align*}
\|C^m_{\bar a,\bar b}x\|&\leq(|\gamma_0|+|\gamma_1|\|p_1\|+|\gamma_2|\|q_1\|+\cdots+|\gamma_s|\|p_n\|\|q_k\|)\|C^m_{\bar a,\bar b}\mathbf 1\|\\&\leq
K_2(|\gamma_0|+|\gamma_1|+|\gamma_2|+\cdots+|\gamma_s|)\|C^m_{\bar a,\bar b}\mathbf 1\|\\&\leq
{K_2}K_1^{-1}\|x\|\|C^m_{\bar a,\bar b}\mathbf 1\|.
\end{align*}
Taking the supremum over all $x$ of norm $1$ we see that
\[\|C^m_{\bar a,\bar b}\|\leq{K_2}K_1^{-1}\|C^m_{\bar a,\bar b}\mathbf 1\|\] holds for each $m$.
So it follows that
\[r_\sigma(L_{\bar a}-R_{\bar b},\mathcal L(X))=\limsup_{m}\|C^m_{\bar a,\bar b}\|^{1/m}\leq\limsup_{m}\|C^m_{\bar a,\bar b}\mathbf 1\|^{1/m}=\varrho(\bar a,\bar b).\]
On the other hand it follows trivially, from $\|C^m_{\bar a,\bar b}\mathbf 1\|\leq\|C^m_{\bar a,\bar b}\|$, that
$\varrho(\bar a,\bar b)\leq r_\sigma(L_{\bar a}-R_{\bar b},\mathcal L(X))$, and hence $\varrho(\bar a,\bar b)=r_\sigma(L_{\bar a}-R_{\bar b},\mathcal L(X))$. But of course $\varrho(\bar a,\bar b)=\varrho(a,b)$.
\end{proof}

\begin{theorem}\label{char} Suppose $\sigma(a)$ and $\sigma(b)$ are finite with $\sigma(a)=\{\lambda_1,\dots,\lambda_n\}$, $\sigma(b)=\{\beta_1,\dots,\beta_k\}$. If $\{p_1,\dots,p_n\}$ and $\{q_1,\dots,q_k\}$ are the corresponding Riesz projections then
\begin{equation}\label{rhogeo}
\varrho(a,b)=\sup\{|\lambda_i-\beta_j|:p_iq_j\not=0\}.
\end{equation}
\end{theorem}
\begin{proof} By Lemma~\ref{spec3} we have that
\begin{equation}\label{rhor}
\varrho(a,b)=r_\sigma\left(\sum_{i=1}^n\lambda_iL_{p_i}-\sum_{i=1}^k\beta_iR_{q_i},\mathcal L(X)\right)
\end{equation}
The preceding formula remains valid if we scale down to the commutative unital subalgebra generated by the $L_{p_i}$ and the $R_{q_i}$. Notice that $\sum_iL_{p_i}=I$, and $\sum_iR_{q_i}=I$. From this, together with the fact that the $L_{p_i}$ are mutually orthogonal, and the $R_{q_i}$ are mutually orthogonal, we now have the following: Corresponding to each $\chi$ belonging to the character space of the algebra there exists a unique pair, say $L_{p_t}$ and $R_{q_s}$, such that $\chi(L_{p_t})=1=\chi(R_{q_s})$ and $\chi(L_{p_i})=0=\chi(R_{q_j})$ whenever $i\not=t, j\not=s$. Conversely, if the product $p_tq_s\not=0$, then the projection $L_{p_t}R_{q_s}\not=0$ and hence there is $\chi$ such that
$\chi(L_{p_t}R_{q_s})=1$. So, for each of the two projections, we have  $\chi(L_{p_t})=1=\chi(R_{q_s})$. With these observations \eqref{rhor} gives the formula \eqref{rhogeo}.
\end{proof}

It is not obvious from \eqref{varrho} that $\varrho$ is not symmetric (see the comments in \cite[p.251]{l+n} regarding this matter). However, Theorem~\ref{char} prescribes the construction of $a,b$ such that $\varrho(a,b)\not=\varrho(b,a)$; the formula \eqref{rhogeo} suggests that one should look for Riesz projections, say $p$ and $q$, such that $pq\not=0$ but $qp=0$:

\begin{example}\label{exa} Let $A$ be the free algebra generated by the alphabet
$\{\mathbf 1,x_1,x_2\}$, subject to the conditions $x_1^2=x_1$, $x_2^2=x_2$, $x_1x_2=0$ and $x_2x_1\not=0$. $A$ is a Banach algebra with
\[\|\alpha_0\mathbf 1+\alpha_1x_1+\alpha_2x_2+\alpha_3x_2x_1\|=\sum\limits_j|\alpha_j|.\]
Now take $a=\frac{1}{2}x_1$ and $b=-\frac{1}{2}x_2$. Then \[C^n_{a,b}\mathbf 1=\frac{1}{2^n}(x_1+x_2)\Rightarrow\|C^n_{a,b}\mathbf 1\|=\frac{1}{2^{n-1}}\Rightarrow\varrho(a,b)=\frac{1}{2}.\]
On the other hand
\begin{align*}
C^n_{b,a}\mathbf 1&=\left(-\frac{1}{2}\right)^n\left[{n \choose 0}x_2+{n \choose 1}x_2x_1+\cdots+{n \choose {n-1}}x_2x_1+{n \choose n}x_1\right]\\&\Rightarrow
\|C^n_{b,a}\mathbf 1\|=\frac{1}{2^n}\sum_{j=0}^n{n \choose j}=1\Rightarrow\varrho(b,a)=1.
\end{align*}
\end{example}
For a more concrete exposition, notice that $A$ in Example~\ref{exa} is isomorphic to a four-dimensional subalgebra of $M_3(\mathbb C)$, the algebra of $3\times 3$ complex matrices.

\begin{theorem}\label{char2} Suppose $\sigma^\prime(a)$ and $\sigma^\prime(b)$ are discrete sets which cluster at $0\in\mathbb C$, if anywhere.
   If $\sigma^\prime(a)=\{\lambda_1,\lambda_2,\dots\}$ and $\sigma^\prime(b)=\{\beta_1,\beta_2,\dots\}$ denote the nonzero spectral points of $a$ and $b$, and if $\{p_1,p_2,\dots\}$ and $\{q_1,q_2\dots\}$ are the corresponding Riesz projections, then $\varrho$ takes at least one of the following values:
\begin{itemize}
\item[(i)]{ $\varrho(a,b)=\sup \{|\lambda_i-\beta_j|:p_iq_j\not=0\}$, or }
\item[(ii)]{ $\varrho(a,b)=|\lambda_i|$ for some $i\in\mathbb N$, or   }
\item[(iii)]{ $\varrho(a,b)=|\beta_i|$ for some $i\in\mathbb N$. }
 \end{itemize}
Moreover, $\varrho(a,b)=0$ if and only if
 the spectra and the corresponding Riesz projections of $a$ and $b$ coincide.
\end{theorem}

\begin{proof} We prove the result where both $\sigma(a)$ and $\sigma(b)$ are infinite sets; the other cases follow similarly: For each $n\in\mathbb N$ let  $a_n=\sum_{i=1}^n\lambda_ip_i$ and $b_n=\sum_{i=1}^n\beta_iq_i$, and put $p_{0,n}=\mathbf 1-\sum_{i=1}^np_i$, $q_{0,n}=\mathbf 1-\sum_{i=1}^nq_i$. As $\sigma(a)$, $\sigma(b)$ are assumed to be infinite, we must have $p_{0,n}\not=0$, $q_{0,n}\not=0$. Note that $\sigma(a_n)=\{\lambda_0,\lambda_1,\dots,\lambda_n\}$ with $\lambda_0=0$ and similarly $\sigma(b_n)=\{\beta_0,\beta_1,\dots,\beta_n\}$ with $\beta_0=0$ (because $a_np_{0,n}=0$ and $b_nq_{0,n}=0$ respectively).  Further, for each $n$, let $\Gamma_{a,n}$ be a simple closed curve, disjoint from $\sigma(a)$, and surrounding only the subset $\{\lambda_{n+1},\lambda_{n+2},\dots\}\cup\{0\}\subset\sigma(a)$. If we notice that, for each $n$,
 \[a=\sum_{i=1}^nap_i+\frac{1}{2\pi i}\int_{\Gamma_{a,n}}\lambda(\lambda\mathbf 1-a)^{-1}\,d\lambda,\] and that $a_n$ commutes with $a$, then it follows that
 $\sigma(a-a_n)\subseteq\{\lambda_{n+1},\lambda_{n+2},\dots\}\cup\{0\}$, and hence $r_{\sigma}(a-a_n)\rightarrow 0$ as $n\rightarrow\infty$. In the same way it follows that  $r_{\sigma}(b-b_n)\rightarrow 0$. Using the triangle inequality for $\varrho$, together with the fact that $\varrho(x,y)=r_{\sigma}(x-y)$ whenever $x$ and $y$ commute, we then obtain
 \[|\varrho(a_n,b_n)-\varrho(a,b)|\leq r_{\sigma}(a-a_n)+r_{\sigma}(b-b_n),\] whence it follows that $\varrho(a,b)=\lim_n\varrho(a_n,b_n)$.
 We now want to use Theorem~\ref{char} to calculate $\varrho(a_n,b_n)$; this requires knowledge of the Riesz projections $p(\lambda_i,a_n)$ and $p(\beta_i,b_n)$ for $i=0,1,\dots, n$:  Observe, for $\lambda\not\in\sigma(a_n)$, that
 \[(\lambda\mathbf 1-a_n)^{-1}=\frac{\mathbf 1}{\lambda}+\sum_{i=1}^n\frac{\lambda_i}{\lambda(\lambda-\lambda_i)}p_i.\] So it follows from the
 Cauchy Integral Formula, and the Cauchy Integral Theorem, that for each $0<i\leq n$, $p(\lambda_i,a_n)=p_i$. A similar argument yields
 $p(\beta_i,b_n)=q_i$ when $0<i\leq n$. It is then obvious that $p(\lambda_0,a_n)=p_{0,n}$ and $p(\beta_0,b_n)=q_{0,n}$.
 Define, for each $n\in\mathbb N$,
\[U_{1,n}=\{|\lambda_i-\beta_j|:p_iq_j\not=0,\,  i,j=1,\dots,n\},\]
\[U_{2,n}=\{|\lambda_i|:p_iq_{0,n}\not=0,\, i=1,\dots,n\},\]
\[U_{3,n}=\{|\beta_i|:p_{0,n}q_i\not=0,\, i=1,\dots,n\},\]
and $U_n=\cup_{j=1}^3U_{j,n}$. If we keep $n$ fixed for the moment, writing  $p_0=p_{0,n}$, $q_0=q_{0,n}$, then, by Theorem~\ref{char}, we obtain
\begin{equation}\label{rhon}
\varrho(a_n,b_n)=\sup\{|\lambda_i-\beta_j|:p_iq_j\not=0, i,j=0,1\dots,n\}=\sup U_n.
 \end{equation}
Notice that $U_n\not=\emptyset$, because if $U_{1,n}=\emptyset$, then, for instance, $p_1q_j=0$ for $j=1,\dots,n$ so that $p_1q_{0,n}=p_1\not=0$ whence $|\lambda_1|\in U_{2,n}\subseteq U_n$. Having established \eqref{rhon}, we are now in a position to derive the conclusion of Theorem~\ref{char2}.
 We shall first prove the statement that $\varrho(a,b)=0$ if and only if  the spectra and the corresponding Riesz projections of $a$ and $b$ coincide: For the reverse implication notice that we can take $a_n=b_n$ for each $n\in\mathbb N$. Thus $\varrho(a,b)=\lim_n\varrho(a_n,b_n)=0$. Suppose, conversely, that $\varrho(a,b)=0$. First let us remark that for each index $i_*$ we can find an index $j_*$ such that $p_{i_*}q_{j_*}\not=0$; if this was not true, i.e. $p_{i_*}q_j=0$ for all $j$, then we may infer that $0\not=p_{i_*}=p_{i_*}q_{0,n}$ for all $n\geq i_*$. But this means that $|\lambda_{i_*}|\in U_{2,n}\subseteq U_n$ for all $n\geq i_*$ which in turn implies $\varrho(a,b)=\lim_n\sup U_n\geq |\lambda_{i_*}|>0$, contradicting $\varrho(a,b)=0$. We therefore have the implication:
\begin{equation}\label{nonempty}
 \varrho(a,b)=0\Rightarrow W:=\{|\lambda_i-\beta_j|:p_iq_j\not=0\}\not=\emptyset.
  \end{equation}
  We proceed to prove $\sigma(a)=\sigma(b)$. Since the spectra of both $a$ and $b$ are infinite, the hypothesis implies $0\in \sigma(a)\cap\sigma(b)$. For a contradiction, suppose that $0\not=\lambda_{i_*}\in\sigma(a)$ but $\lambda_{i_*}\notin\sigma(b)$. Then, as in the paragraph preceding
\eqref{nonempty}, we can find an index $j_*$ such that $p_{i_*}q_{j_*}\not=0$. If $n\geq\max\{i_*,j_*\}$ is arbitrary, then $|\lambda_{i_*}-\beta_{j_*}|\in U_{1,n}\subseteq U_n$ from which
\[\varrho(a,b)=\lim_n\sup U_n\geq|\lambda_{i_*}-\beta_{j_*}|\geq\dist(\lambda_{i_*},\sigma(b))>0\]
giving the required contradiction. Therefore $\sigma(a)\subseteq\sigma(b)$. Similarly $\sigma(b)\subseteq\sigma(a)$, and we have   $\sigma(a)=\sigma(b)$.
 It remains to show that the Riesz projections,
$p(\lambda_{i_*},a)=:p_{i_*}$ and $p(\lambda_{i_*},b)=:q_{i_*}$, corresponding to a common nonzero spectral value $\lambda_{i_*}\in\sigma(a)=\sigma(b)$, are in fact equal: First observe that $\sup W=0$; indeed if for some indices $i_*,j_*$ we have
$0\not=|\lambda_{i_*}-\lambda_{j_*}|\in W$, then  $|\lambda_{i_*}-\lambda_{j_*}|\in U_{1,n}$ for all $n\geq\max\{i_*,j_*\}$, and hence, as before,
$\varrho(a,b)>0$ which is absurd. If we fix an index $i_*$, then $p_{i_*}q_j=0$ whenever $j\not=i_*$ because otherwise $p_{i_*}q_j\not=0$ implies
 $|\lambda_{i_*}-\lambda_j|\in W$, forcing $\lambda_{i_*}=\lambda_j$, which is possible only if $j=i_*$ (as the points in the spectrum are distinct).
 Therefore
\begin{equation*}
p_{i_*}-p_{i_*}q_{i_*}=p_{i_*}q_{0,n} =p_{i_*}\left(\mathbf 1-\sum_{j=1}^nq_j\right)\mbox{ for all }n\geq i_*.
\end{equation*}
Now if $p_{i_*}\not=p_{i_*}q_{i_*}$, then $\varrho(a_n,b_n)\geq|\lambda_{i_*}|$ for all $n\geq i_*$, which again leads to
$\varrho(a,b)\geq|\lambda_{i_*}|>0$. So we conclude that $p_{i_*}=p_{i_*}q_{i_*}$.
A similar argument, using the sets $U_{3,n}$ instead, gives $q_{i_*}=p_{i_*}q_{i_*}$, and thus $p_{i_*}=q_{i_*}$. We have now shown that
$\varrho(a,b)=0$ if and only if the spectra and the corresponding Riesz projections of $a$ and $b$ coincide.

For the remaining part of the statement: If $\varrho(a,b)=0$, then \eqref{nonempty} says $W\not=\emptyset$, and, as we have shown,
$\sup W=0$; hence (i) is valid. Suppose that $\varrho(a,b)>0$ and that $\sup W<\lim_n\sup U_n$ (if $W=\emptyset$ we let $\sup W=0$). If we set
$\tau_n=\sup(U_{2,n}\cup U_{3,n})$, then $\lim_n\sup U_n=\lim_n\tau_n$ whence it follows that there exists $N\in\mathbb N$ such that
$\tau_n>\sup W$ for all $n\geq N$. In particular, we can build either a sequence $\left(\lambda_{i,{n_k}}\right)$ whose members belong to $\sigma^\prime(a)$, or a sequence $\left(\beta_{j,{n_k}}\right)$ whose members belong to $\sigma^\prime(b)$, such that $\left|\lambda_{i,{n_k}}\right|=\tau_{n_k}$ or $\left|\beta_{j,{n_k}}\right|=\tau_{n_k}$, and $\lim_k\left|\lambda_{i,{n_k}}\right|=\lim_n\sup U_n$ or $\lim_k\left|\beta_{j,{n_k}}\right|=\lim_n\sup U_n$. To avoid trivial misunderstanding, the notation indicates that these sequences are not subsequences of, respectively, $(\lambda_i)$ and $(\beta_j)$ but, rather, sequences constructed by extracting individual members of the sets $\sigma^\prime(a)$ and $\sigma^\prime(b)$ (i.e. repetition of terms may occur). Anyhow, if we assume the existence of the sequence $\left(\lambda_{i,{n_k}}\right)$ satisfying the aforementioned properties, then, since $\lim_n\sup U_n>0$, it follows that the sequence $(|\lambda_{i,{n_k}}|)$ must eventually be constant (because the spectrum of $a$ clusters only at $0\in\sigma(a)$). This means there exists
an index, $i_*$, such that $\lim_n\sup U_n=|\lambda_{i_*}|$ and hence that $\varrho(a,b)=|\lambda_{i_*}|$, so that (ii) holds. If the sequence
$\left(\lambda_{i,{n_k}}\right)$ cannot be found, then a similar argument, with the sequence $\left(\beta_{j,{n_k}}\right)$, shows that (iii) holds.
\end{proof}

For elements $a,b\in A$ satisfying the hypothesis of Theorem~\ref{char2} it follows that $\varrho(a,b)=0\Leftrightarrow\varrho(b,a)=0$ which simplifies the requirement for quasinilpotent equivalence. The proof of Theorem~\ref{char2} also establishes a formula for $\varrho$: Let us assume the hypothesis of Theorem~\ref{char2}, where both $\sigma(a)$ and $\sigma(b)$ are infinite sets. Define, as in the proof of Theorem~\ref{char2},
\[W:=\{|\lambda_i-\beta_j|:p_iq_j\not=0\}.\]
If $W=\emptyset$, then the proof of Theorem~\ref{char2} shows that, for each $n$, we have $\varrho(a_n,b_n)=\sup\{r_\sigma(a_n),r_\sigma(b_n)\}$. Therefore \[\varrho(a,b)=\lim_n\varrho(a_n,b_n)=\lim_n\sup\{r_\sigma(a_n),r_\sigma(b_n)\}=\sup\{r_\sigma(a),r_\sigma(b)\}.\]
Suppose now $W\not=\emptyset$. If for some index $k$ we have $|\lambda_k|>\sup W$ then, since $\lim_j\beta_j=0$, there exists $N>0$ such that $p_kq_j=0$ for all $j\geq N$; if this was not true then some subsequence, say $(q_{j_m})$, of  $(q_j)$ satisfies $p_kq_{j_m}\not=0$ for each $m$. But then, by definition, $|\lambda_k-\beta_{j_m}|\in W$ for each $m$. Letting $m\rightarrow\infty$, so that $\beta_{j_m}\rightarrow0$, we see that $\sup W\geq |\lambda_k|$ contradicting the assumption. So for any index $k$ satisfying $|\lambda_k|>\sup W$ we have that $\lim_n\sum_{j=1}^np_kq_j=:\sum_{j=1}^\infty p_kq_j$ exists in $A$. Moreover, in the same way we can prove that if $|\beta_k|>\sup W$ then $\sum_{j=1}^\infty p_jq_k$ exists in $A$. Thus, if $W\not=\emptyset$,
we may define:
\begin{gather*}
W_{\lambda}:=\left\{|\lambda_k|:|\lambda_k|>\sup W\ \mbox{\rm and } \sum_{j=1}^\infty p_kq_j\not=p_k\right\},\\
W_{\beta}:=\left\{|\beta_k|:|\beta_k|>\sup W\ \mbox{\rm and } \sum_{j=1}^\infty p_jq_k\not=q_k\right\}.
\end{gather*}
The arguments leading to Theorem~\ref{char2} now proves the following formula:
\begin{theorem}[global spectral formula for $\varrho$]\label{charf} With the hypothesis of Theorem~\ref{char2} (where both $\sigma(a)$ and $\sigma(b)$ are infinite sets), we have
 \begin{displaymath}
\varrho(a,b)=\left\{\begin{array}{cc}\sup W\cup W_{\lambda}\cup W_{\beta} & \mbox{ if }W\not=\emptyset\\
\sup\{r_\sigma(a),r_\sigma(b)\} & \mbox{ if }W=\emptyset.\end{array}\right.
\end{displaymath}
\end{theorem}

We may remark that if both $\sigma(a)$ and $\sigma(b)$ are finite sets then the formula in Theorem~\ref{char} applies. If one spectrum is infinite
($\sigma(a)$), and the other finite ($\sigma(b)$), then one can easily adjust the formula in Theorem~\ref{char2}: Specifically, if $\sigma(b)$ is finite, then every spectral value has a corresponding Riesz projection whence the set $W_\lambda$ becomes redundant with its role being taken over by an adjusted version of the set $W$ (where $q_0$ is the Riesz projection corresponding to $\beta_0=0$). To deal with the cluster point $0\in\sigma(a)$ one needs a limiting process, as in the proof of Theorem~\ref{char2}, which necessitates the definition of $W_\beta$.

To illustrate the implementation as well as the practical value of Theorem~\ref{charf} consider the following:

\begin{example} With the usual understanding, let $X$ be the Banach space $L^1[1,\infty)$. Given $f\in X$ define noncommuting $T,S\in\mathcal L(X)$ by \[(Tf)(t)=\frac{f(t)}{k}\ \ \ \mbox{ if }t\in[k,k+1),\ k\in\mathbb N\] and

\begin{displaymath}
(Sf)(t)=\left\{\begin{array}{cc}f(t) & \mbox{ if }t\in[1,2)\\
{\bigl[f(t)+f(t-k+1)\bigr]}\big/{k^2} & \mbox{ if }t\in[k,k+1),\ 1<k\in\mathbb N.\end{array}\right.
\end{displaymath}
\smallskip

 It is straightforward to calculate $\sigma(T)=\{\frac{1}{k}:k\in\mathbb N\}\cup\{0\}$, and $\sigma(S)=\{\frac{1}{k^2}:k\in\mathbb N\}\cup\{0\}$.
Write $p\left(\frac{1}{k},T\right)=:P_k$ and $p\left(\frac{1}{k^2},S\right)=:Q_k$. If $k\in\mathbb N$ and $f\in X$ then it follows readily, by Cauchy's formula, that:
\begin{enumerate}
\item{ $(P_kf)(t)=\chi_{[k,k+1)}(t)f(t)$,}
\item{ $Q_1=P_1$,}
\item{ $(Q_kf)(t)=\chi_{[k,k+1)}(t)\displaystyle{\left[f(t)+\frac{f(t-k+1)}{1-k^2}\right]}$ \ \ $(k\not=1)$.}
\end{enumerate}
Then $P_kQ_l=Q_l$ if $k=l$, and $P_kQ_l=0$ if $k\not=l$. In terms of Theorem~\ref{charf}, we observe
that $W=\{\frac{1}{k}-\frac{1}{k^2}:k\in\mathbb N\}$, $W_\lambda=\{\frac{1}{2},\frac{1}{3}\}$, and $W_\beta=\emptyset$. Thus
$\varrho(T,S)=\frac{1}{2}$. Also, $Q_kP_l=P_l$ if $k=l$, and $Q_kP_l=0$ if $k\not=l$ implies that $\varrho(S,T)=\frac{1}{2}$. So
$\rho(T,S)=\frac{1}{2}.$

\end{example}

\bibliographystyle{amsplain}

\end{document}